\documentclass[11pt,preprint]{elsarticle}
\journal{Journal of Mathematical Analysis and Applications}

\usepackage{amsmath,amssymb,amsfonts,amsthm,graphicx,hyperref,color}
\usepackage[labelfont=bf]{caption}
\usepackage{booktabs}
\usepackage{natbib}
\usepackage{lmodern} 

\usepackage{geometry} 
\newtheorem{theorem}{Theorem}[section]

\newtheorem{lemma}{Lemma}[section]
\newtheorem{corollary}{Corollary}[section]

\newtheorem{remark}{Remark}[section]
\newtheorem*{notation}{Notation}
\numberwithin{equation}{section}

\newcommand{\N}{\mathbb{N}}
\newcommand{\R}{\mathbb{R}}
\newcommand{\C}{\mathbb{C}}

\newcommand{\EE}{\mathsf{E}} 
\newcommand{\bb}[1]{\boldsymbol{#1}}
\newcommand{\rd}{\mathrm{d}}
\newcommand{\e}{\varepsilon}
\newcommand{\tr}{\mathrm{tr}}
\newcommand{\etr}{\mathrm{etr}}
\newcommand{\sbullet}{\scalebox{0.65}{$\bullet$}}

\begin{document}

\begin{frontmatter}

\title{On noncentral Wishart mixtures of noncentral Wisharts and \\
their use for testing random effects in factorial design models}

\author[a1]{Christian Genest}
\author[a2]{Anne MacKay}
\author[a1,a2,a3]{Fr\'ed\'eric Ouimet\corref{mycorrespondingauthor}}

\address[a1]{Department of Mathematics and Statistics, McGill University, Canada}
\address[a2]{D\'epartement de math\'ematiques, Universit\'e de Sherbrooke, Canada}
\address[a3]{D\'epartement de math\'ematiques et d'informatique, Universit\'e du Qu\'ebec \`a Trois-Rivi\`eres, Canada}

\cortext[mycorrespondingauthor]{Corresponding author. Email address: frederic.ouimet2@uqtr.ca}

\begin{abstract}
It is shown that a noncentral Wishart mixture of noncentral Wishart distributions with the same degrees of freedom yields a noncentral Wishart distribution, thereby extending the main result of Jones and Marchand [\textit{Stat}~\textbf{10} (2021), Paper No.~e398, 7~pp.] 
from the chi-square to the Wishart setting. To illustrate its use, this fact is then employed to derive the finite-sample distribution of test statistics for random effects in a two-factor factorial design model with $d$-dimensional normal data, thereby broadening the findings of Bilodeau [\textit{ArXiv} (2022), 6~pp.], who treated the case $d = 1$. 
The same approach makes it possible to test random effects in more general factorial design models.
\end{abstract}


\end{frontmatter}

\section{Introduction}\label{sec:intro}

The Wishart distribution has long been a cornerstone of multivariate statistical analysis, underpinning methods for covariance estimation, hypothesis testing, and the analysis of quadratic forms arising from normally distributed data. With the growing need for flexible multivariate models, coupled with the demand for explicit distributional forms to facilitate inference and goodness-of-fit testing, mixtures of Wishart distributions have proven to be an effective tool to capture randomness in dependence structures. Their applicability spans diverse areas, including signal modeling in diffusion-weighted magnetic resonance imaging \citep{doi:10.1109/TMI.2007.907552, doi:10.1016/j.neuroimage.2007.03.074, MR2596585, doi:10.1007/978-3-319-61358-1_12}, the characterization of noise in density estimation on the cone of positive definite matrices in a deconvolution framework \citep{MR3012414, MR2838725}, clustering of time series data \citep{MR4524855}, the distribution of polarimetric synthetic aperture radar (PSAR) data \citep{doi:10.1109/TGRS.2016.2590145}, classification algorithms for PSAR data \citep{doi:10.1080/01431161.2022.2054293}, and the development of improved estimators for the inverse scatter matrix using unbiased risk estimation \citep{MR4260037}.

Of interest in this paper is a specific class of Wishart mixtures, where both the mixing and the mixed distributions are either Wishart or noncentral Wishart. Such mixtures naturally arise in the context of random-effect modeling of normal data within two-factor factorial designs, as detailed in the MANOVA application of Section~\ref{sec:application}, where the sum of outer-products associated with each factor follows a noncentral Wishart distribution, and the noncentrality matrix parameter itself is Wishart-distributed. After introducing the necessary definitions, notations, and background on multivariate statistics in Section~\ref{sec:definitions}, we show in Section~\ref{sec:results} that a noncentral Wishart mixture of noncentral Wishart distributions sharing the same degrees of freedom remains a noncentral Wishart distribution.

Our result extends to arbitrary dimension $d \geq 1$ an earlier finding of \citet{MR4319008}, who showed that a noncentral chi-square mixture of noncentral chi-square distributions with the same degrees of freedom is a noncentral chi-square distribution. The extension of this closure property to the more general matrix-variate case is significant in two ways. First, from a theoretical perspective, it deepens our understanding of the structural properties of the Wishart family. Second, and perhaps more importantly, it opens up new avenues for goodness-of-fit testing in multivariate analysis. In the aforementioned factorial design framework, for instance, derivations based on our results lead to exact finite-sample distributions for certain test statistics, providing a rigorous approach for testing the significance of $d$-dimensional factors with Gaussian random effects without relying on asymptotic approximations. The findings of \citet{arXiv:2103.13202}, who treated the case $d = 1$, are thereby extended to a fully multivariate context.

\section{Definitions and notations}\label{sec:definitions}

For any integer $d \in \mathbb{N} = \{1, 2, \ldots \}$, let $\mathcal{S}^d$, $\mathcal{S}_+^d$, and $\mathcal{S}_{++}^d$, be the spaces of real symmetric, nonnegative definite, and positive definite, matrices of size $d \times d$, respectively. Let the operator $\tr(\cdot)$ stand for the trace, $\mathrm{etr}(\cdot) = \exp\{\tr(\cdot)\}$ the exponential trace, and $|\cdot|$ the determinant.

The multivariate gamma function~$\Gamma_d$ is defined, for any $\beta \in \C$ such that $\textrm{Re}(\beta) > (d-1)/2$, by
\[
\Gamma_d(\beta) = \int_{\mathcal{S}_{++}^d} \etr(- X)  |X|^{\beta - (d + 1)/2} \rd X,
\]
which is a natural generalization to the cone $\mathcal{S}_{++}^d$ of the classical gamma function.

For any degree-of-freedom parameter $\nu \in (d-1,\infty)$, any scale matrix $\Sigma \in \mathcal{S}_{++}^d$, and any noncentrality matrix parameter $\Theta\in \R^{d\times d}$ such that $\Theta \Sigma^{-1}\in \mathcal{S}_+^d$, the probability density function of the corresponding noncentral Wishart distribution, henceforth denoted $\mathcal{W}_d(\nu,\Sigma,\Theta)$, is given, for all matrices $X \in \mathcal{S}_{++}^d$, by
\[
f_{\nu,\Sigma,\Theta}(X) = \frac{|X|^{\nu/2 - (d + 1)/2} \etr(-\Sigma^{-1}X/2)}{|2 \Sigma|^{\nu/2} \Gamma_d(\nu/2)} \etr(-\Theta/2)  {}_0F_1(\nu/2; \Theta\Sigma^{-1}X/4),
\]
where ${}_0F_1$ denotes a hypergeometric function of matrix argument; see, e.g., \citet[Definition~7.3.1]{MR652932}. Whenever a random matrix $\mathfrak{X}$ of size $d \times d$ follows this distribution, one writes $\mathfrak{X}\sim \mathcal{W}_d(\nu,\Sigma,\Theta)$ for short. When $\Theta = 0_{d \times d}$, one recovers the (central) Wishart distribution, and one writes $\mathfrak{X}\sim \mathcal{W}_d(\nu,\Sigma)$, omitting the noncentrality parameter. For additional information concerning the Wishart distribution, see, e.g., Section~3.2.1 of the book by~\citet{MR652932}.

\begin{remark}\label{rem:mgf}
By Theorem~10.3.3 of \citet{MR652932}, the moment generating function of $\mathfrak{X}\sim \mathcal{W}_d(\nu,\Sigma,\Theta)$ is given, for all $T\in \mathcal{S}^d$ such that $\Sigma^{-1} - 2 T\in \mathcal{S}_{++}^d$, by
\[
M_{\mathfrak{X}}(T)
\equiv \EE\{\etr(T \mathfrak{X})\}
= \frac{\etr\{T \Sigma(I_d - 2 T \Sigma)^{-1} \Theta\}}{|I_d - 2 T \Sigma|^{\nu/2}}.
\]
Alternatively, this expression can also be deduced from Theorem~3.5.3 of \citet{Gupta_Nagar_1999}, where the products $T \Sigma$ are incorrectly written as $\Sigma T$ due to a misordering of the matrices in the exponential trace near the end of their proof. This typo is consequential for our calculations and caused many headaches!
\end{remark}

\begin{remark}\label{rem:normal}
Given any $\nu\in \N$, $M\in \R^{\nu\times d}$, and $\Sigma\in \mathcal{S}_{++}^d$, the random matrix $\mathfrak{N}$ is said to have a matrix-variate normal distribution, written $\mathfrak{N}\sim \mathcal{N}_{\nu\times d}(M, I_{\nu} \otimes \Sigma)$, if $\mathrm{vec}(\mathfrak{N})\sim \mathcal{N}_{\nu d}\{\mathrm{vec}(M), I_{\nu} \otimes \Sigma\}$, where $\mathrm{vec}(\cdot)$ is the vectorization operator, $\mathcal{N}_{\nu d}$ is the $(\nu d)$-dimensional multivariate normal distribution, and $\otimes$ denotes the Kronecker product. It is well known that if $\mathfrak{N}\sim \mathcal{N}_{\nu\times d}(M, I_{\nu} \otimes \Sigma)$ with $\nu\geq d$, then $\mathfrak{N}^{\top}\mathfrak{N}\sim \mathcal{W}_d(\nu, \Sigma, \Sigma^{-1} M^{\top} M)$; see, e.g., \citet[Theorem~10.3.2]{MR652932}.
\end{remark}

The lemma below extends Theorem~10.3.5 of \citet{MR652932} from the matrix-variate normal setting to the Wishart setting. This result is most likely known but we were unable to find an explicit statement, so it is given below for completeness.

\begin{lemma}\label{lem:conjugation}
Let $\nu \in (d-1,\infty)$, $C$, $\Sigma \in \mathcal{S}_{++}^d$ and $\Theta\in \R^{d\times d}$ be given such that $\Theta \Sigma^{-1}\in \mathcal{S}_+^d$. If $\mathfrak{X}\sim \mathcal{W}_d(\nu, \Sigma, \Theta)$, then
\[
C \mathfrak{X} C\sim \mathcal{W}_d(\nu, C \Sigma C, C^{-1} \Theta C).
\]
\end{lemma}

\begin{proof}
Using Remark~\ref{rem:mgf} and the invariance of the trace under cyclic permutations, note that, for every $T\in \mathcal{S}^d$ satisfying $\Sigma^{-1} - C T C\in \mathcal{S}_{++}^d$,
\[
M_{C \mathfrak{X} C}(T)
= M_{\mathfrak{X}}(C T C)
= \frac{\etr\{C T C \Sigma C C^{-1} (I_d - 2 C T C \Sigma)^{-1} C C^{-1} \Theta\}}{|I_d - 2 C T C \Sigma|^{\nu/2}}.
\]
Given that $C^{-1} (I_d - 2 C T C \Sigma)^{-1} C = \{C^{-1} (I_d - 2 C T C \Sigma) C\}^{-1} = (I_d - 2 T C \Sigma C)^{-1}$, one has
\[
M_{C \mathfrak{X} C}(T)
= \frac{\etr\{T C \Sigma C (I_d - 2 T C \Sigma C)^{-1} C^{-1} \Theta C\}}{|I_d - 2 T C \Sigma C|^{\nu/2}}.
\]
Using Remark~\ref{rem:mgf} again proves the claim of the lemma.
\end{proof}

\begin{notation}
For any matrix $R \in \R^{d \times d}$ and $P \in \mathcal{S}_+^d$, define the conjugation
\[
R_P = P^{1/2} R P^{1/2},
\]
where $P^{1/2}$ denotes the symmetric square root of $P$. This definition extends naturally to random matrices, i.e., $\mathfrak{R}_P = P^{1/2} \mathfrak{R} P^{1/2}$. This notation is used from hereon unless it leads to cumbersome sub-subscripts, in which case the expression is written explicitly. For example, Lemma~\ref{lem:conjugation} can be restated as follows: if $\mathfrak{X} \sim \mathcal{W}_d(\nu, \Sigma, \Theta)$, then $\mathfrak{X}_V \sim \mathcal{W}_d(\nu, \Sigma_V, V^{-1} \Theta_V)$ with $V = C^2$.
\end{notation}

\section{Results}\label{sec:results}

This paper's main result, Theorem~\ref{thm:main} below, states that a noncentral Wishart mixture of noncentral Wishart distributions --- where the degrees of freedom of both noncentral Wisharts are the same --- is itself a noncentral Wishart distribution with the same degrees of freedom. It generalizes the main finding (Result~1) of \citet{MR4319008} in two ways: by extending the chi-square setting ($d = 1$) to the Wishart setting ($d\geq 1$), and by allowing arbitrary scale parameters ($A,\Sigma\in \mathcal{S}_{++}^d$) for both the mixing and mixed random matrices.

\begin{theorem}\label{thm:main}
Let $\nu \in (d-1,\infty)$, $A$, $\Sigma$, $H \in \mathcal{S}_{++}^d$ and $\Delta\in \mathcal{S}_+^d$ be given. If
\[
\mathfrak{X} \mid \{\mathfrak{Y} = Y\} \sim \mathcal{W}_d(\nu, A, A^{-1/2} Y_H A^{1/2}), \quad \mathfrak{Y} \sim \mathcal{W}_d(\nu, \Sigma, \Sigma^{-1} \Delta),
\]
then $\mathfrak{X}\sim \mathcal{W}_d(\nu, V, V^{-1} A^{1/2} \Delta_H A^{1/2})$ with $V = A^{1/2} (I_d + \Sigma_H) A^{1/2}$.
\end{theorem}

\begin{remark}
In Theorem~\ref{thm:main}, the conclusion can also be stated, using Lemma~\ref{lem:conjugation}, as
\[
V^{-1/2} \mathfrak{X} V^{-1/2}\sim \mathcal{W}_d \{\nu, I_d, V^{1/2} (V^{-1} A^{1/2} \Delta_H A^{1/2}) V^{-1/2}\} \equiv \mathcal{W}_d(\nu, I_d, V^{-1/2} A^{1/2} \Delta_H A^{1/2} V^{-1/2}).
\]
Letting the scale matrices $A$ and $\Sigma$ be trivial, i.e., $A = \Sigma = I_d$, the above reduces to
\[
(I_d + H)^{-1/2} \mathfrak{X} (I_d + H)^{-1/2}\sim \mathcal{W}_d\{\nu, I_d, (I_d + H)^{-1/2} \Delta_H (I_d + H)^{-1/2}\}.
\]
When $d = 1$, Result~1 of \citet{MR4319008} for noncentral chi-square mixtures is recovered, viz., if $h \in (0, \infty)$, $\delta \in [0, \infty)$, $X \mid \{Y = y\} \sim \chi^2_{\nu}(y h)$ and $Y\sim \chi_{\nu}^2 (\delta)$, then $X / (1 + h)\sim \chi_{\nu}^2\{h \delta / (1 + h)\}$.
\end{remark}

\begin{proof}[First proof of Theorem~\ref{thm:main}]
By conditioning on the value of $\mathfrak{Y}$, applying Remark~\ref{rem:mgf}, and using the invariance of the trace operator under cyclic permutations, one has, for every $T\in \mathcal{S}^d$ with $A^{-1} - 2 T\in \mathcal{S}_{++}^d$,
\[
\begin{aligned}
M_{\mathfrak{X}}(T)
&= \EE\big[\EE\{\etr(T \mathfrak{X}) \mid \mathfrak{Y}\}\big] \\
&= \EE\left[\frac{\etr\{T A (I_d - 2 T A)^{-1} A^{-1/2} \mathfrak{Y}_H A^{1/2}\}}{|I_d - 2 T A|^{\nu/2}}\right] \\
&= \EE\left[\frac{\etr\{A^{1/2} T A^{1/2} A^{1/2} (I_d - 2 T A)^{-1} A^{-1/2} \mathfrak{Y}_H\}}{|I_d - 2 A^{1/2} T A^{1/2}|^{\nu/2}}\right].
\end{aligned}
\]

Given that $A^{1/2} (I_d - 2 T A)^{-1} A^{-1/2} = (I_d - 2 A^{1/2} T A^{1/2})^{-1} = (I_d - 2 T_A)^{-1} $, one deduces that
\[
M_{\mathfrak{X}}(T)
= \EE\left[\frac{\etr\{T_A (I_d - 2 T_A)^{-1} \mathfrak{Y}_H\}}{|I_d - 2 T_A|^{\nu/2}}\right].
\]
Note that in view of Lemma~\ref{lem:conjugation},
\[
\mathfrak{Y}_H \equiv H^{1/2} \mathfrak{Y} H^{1/2} \sim \mathcal{W}_d\{\nu, \Sigma_H, (\Sigma_H)^{-1} \Delta_H\}.
\]

By applying Remark~\ref{rem:mgf} with $\widetilde{T} \equiv T_A (I_d - 2 T_A)^{-1}$ and using the invariance of the trace operator under cyclic permutations, one finds that, for every $T\in \mathcal{S}^d$ satisfying $A^{-1} - 2 T\in \mathcal{S}_{++}^d$ and $(\Sigma_H)^{-1} - 2 \widetilde{T}\in \mathcal{S}_{++}^d$,
\[
\begin{aligned}
M_{\mathfrak{X}}(T)
&= |I_d - 2 T_A|^{-\nu/2} M_{\mathfrak{Y}_H}(\widetilde{T}) \\
&= |I_d - 2 T_A|^{-\nu/2} \frac{\etr\{\widetilde{T} \Sigma_H (I_d - 2 \widetilde{T} \Sigma_H)^{-1} (\Sigma_H)^{-1} \Delta_H\}}{|I_d - 2 \widetilde{T} \Sigma_H|^{\nu/2}} \\
&= |I_d - 2 T_A|^{-\nu/2} \frac{\etr\{\widetilde{T} (I_d - 2 \Sigma_H \widetilde{T})^{-1} \Delta_H\}}{|I_d - 2 \Sigma_H \widetilde{T}|^{\nu/2}}.
\end{aligned}
\]

Next, notice that
\[
\begin{aligned}
I_d - 2 \Sigma_H \widetilde{T}
&= I_d - 2 \Sigma_H T_A (I_d - 2 T_A)^{-1} \\
&= \{(I_d - 2 T_A) - 2 \Sigma_H T_A\} (I_d - 2 T_A)^{-1} \\
&= \{I_d - 2 (I_d + \Sigma_H) T_A\} (I_d - 2 T_A)^{-1}.
\end{aligned}
\]
Therefore, for every $T\in \mathcal{S}^d$ satisfying $A^{-1} - 2 T\in \mathcal{S}_{++}^d$ and $(\Sigma_H)^{-1} - 2 \widetilde{T}\in \mathcal{S}_{++}^d$, one has
\[
\begin{aligned}
M_{\mathfrak{X}}(T)
&= \frac{\etr\big[T_A \{I_d - 2 (I_d + \Sigma_H) T_A\}^{-1} \Delta_H\big]}{|I_d - 2 (I_d + \Sigma_H) T_A|^{\nu/2}} \\
&= \frac{\etr\big[T_A \{(I_d + \Sigma_H) A^{1/2}\} \{I_d - 2 T A^{1/2} (I_d + \Sigma_H) A^{1/2}\}^{-1} \{(I_d + \Sigma_H) A^{1/2}\}^{-1} A^{-1/2} A^{1/2} \Delta_H\big]}{|I_d - 2 T A^{1/2} (I_d + \Sigma_H) A^{1/2}|^{\nu/2}} \\
&= \frac{\etr\big[T V (I_d - 2 T V)^{-1} V^{-1} A^{1/2} \Delta_H A^{1/2}\big]}{|I_d - 2 T V|^{\nu/2}},
\end{aligned}
\]
where $V = A^{1/2} (I_d + \Sigma_H) A^{1/2}$. Invoking Remark~\ref{rem:mgf} again, together with the analyticity of the moment-generating function, one can verify the claim.
\end{proof}

Using the relationship between the matrix-variate normal and noncentral Wishart distributions described in Remark~\ref{rem:normal} when the degree-of-freedom parameter $\nu$ is an integer, one obtains the following second proof of Theorem~\ref{thm:main} in that case.

\begin{proof}[Second proof of Theorem~\ref{thm:main}, assuming that $\nu\geq d$ is an integer]
Let $M\in \R^{\nu\times d}$ be any matrix such that $M^{\top} M = \Delta$. Consider
\[
\mathfrak{Z} \mid \{\mathfrak{W} = W\}\sim \mathcal{N}_{\nu\times d}(W H^{1/2} A^{1/2}, I_{\nu} \otimes A), \quad \mathfrak{W}\sim \mathcal{N}_{\nu\times d}(M, I_{\nu} \otimes \Sigma).
\]

Using Remark~\ref{rem:normal} and the notation $Y = W^{\top} W$, one recovers the Wishart setting in the statement of the theorem, viz.
\[
\mathfrak{X} \mid \{\mathfrak{W} = W\} = \mathfrak{Z}^{\top} \mathfrak{Z} \mid \{\mathfrak{W} = W\}\sim \mathcal{W}_d(\nu, A, A^{-1/2} Y_H A^{1/2}), \quad \mathfrak{Y} = \mathfrak{W}^{\top} \mathfrak{W}\sim \mathcal{W}_d(\nu, \Sigma, \Sigma^{-1} \Delta).
\]

By Theorem~2.3.10 of \citet{Gupta_Nagar_1999}, note that
\[
\mathfrak{W} H^{1/2} A^{1/2}\sim \mathcal{N}_{\nu\times d}(M H^{1/2} A^{1/2}, I_{\nu} \otimes A^{1/2} \Sigma_H A^{1/2}).
\]
Hence, the above yields $(\mathfrak{Z} - \mathfrak{W} H^{1/2} A^{1/2}) \mid \{\mathfrak{W} = W\}\sim \mathcal{N}_{\nu\times d}(0_{\nu\times d}, I_{\nu} \otimes A)$, which implies that
\[
(\mathfrak{Z} - \mathfrak{W} H^{1/2} A^{1/2})\sim \mathcal{N}_{\nu\times d}(0_{\nu\times d}, I_{\nu} \otimes A) \quad \text{and} \quad (\mathfrak{Z} - \mathfrak{W} H^{1/2} A^{1/2})\perp \mathfrak{W}.
\]

The sum of two independent matrix-variate normals is still matrix-variate normal (this is a trivial consequence of the definition of the matrix-variate normal distribution in Remark~\ref{rem:normal} seen through the lens of the vectorization operation), so one deduces
\[
\mathfrak{Z} = (\mathfrak{Z} - \mathfrak{W} H^{1/2} A^{1/2}) + \mathfrak{W} H^{1/2} A^{1/2}\sim \mathcal{N}_{\nu\times d}\big[M H^{1/2} A^{1/2}, I_{\nu} \otimes \{A^{1/2} (I_d + \Sigma_H) A^{1/2}\}\big].
\]
It follows that
\[
\mathfrak{X} = \mathfrak{Z}^{\top} \mathfrak{Z}\sim \mathcal{W}_d(\nu, V, V^{-1} A^{1/2} \Delta_H A^{1/2}),
\]
where $V = A^{1/2} (I_d + \Sigma_H) A^{1/2}$. This concludes the proof.
\end{proof}

In analogy with the central chi-square case mentioned in Corollary~1 of \citet{MR4319008}, the special case of zero noncentrality parameter ($\Delta = 0_{d \times d}$) in Theorem~\ref{thm:main} implies that a central Wishart mixture of noncentral Wisharts is distributed as a central Wishart. This is formally stated below.

\begin{corollary}\label{cor:main}
Let $\nu \in (d-1,\infty)$ and $A,\Sigma,H\in \mathcal{S}_{++}^d$ be given. If
\[
\mathfrak{X} \mid \{\mathfrak{Y} = Y\} \sim \mathcal{W}_d(\nu, A, A^{-1/2} Y_H A^{1/2}), \qquad \mathfrak{Y} \sim \mathcal{W}_d(\nu,\Sigma),
\]
then $\mathfrak{X}\sim \mathcal{W}_d\{\nu, A^{1/2} (I_d + \Sigma_H) A^{1/2}\}$.
\end{corollary}

\section{Application: Testing random effects in multivariate factorial design models}\label{sec:application}

\subsection{Methodology}\label{sec:methodology}

This section illustrates the practical implications of the theoretical results derived in Section \ref{sec:results} by applying them to the analysis of a two-factor factorial design model with multivariate observations. In the univariate setting, \citet{arXiv:2103.13202} showed that the standard $F$ statistics used to test for fixed effects in a two‐factor factorial design for normal data remain valid when random effects are introduced. We extend this result to higher dimensions through Corollary~\ref{cor:main}. While the discussion is restricted to the case of two factors for the sake of clarity, the extension to models with any number of factors is straightforward.

Generalizing Section~5.3.1 of \citet{MR2552961} to a multivariate framework, consider the two-factor factorial design model with fixed effects
\begin{equation}\label{eq:model}
\bb{Y}_{ijk} = \bb{\mu} + \bb{\alpha}_i + \bb{\beta}_j + (\bb{\alpha\beta})_{ij} + \bb{\e}_{ijk}, \quad i\in \{1,\ldots,a\}, ~j\in \{1,\ldots,b\}, ~k\in \{1,\ldots,n\},
\end{equation}
where $\bb{\mu}$, $\bb{\alpha}_i$, $\bb{\beta}_j$, and $(\bb{\alpha\beta})_{ij}$ are fixed $d$-dimensional vectors, and the errors are normally distributed and independent, i.e., $\smash{\bb{\e}_{ijk} \stackrel{\mathrm{iid}}{\sim} \mathcal{N}_d(\bb{0}_d, \Sigma)}$ for some scale matrix $\Sigma\in \mathcal{S}_{++}^d$. To ensure that the model parameters are identifiable, the following constraints are imposed for all $(i,j)\in \{1,\ldots,a\} \times \{1,\ldots,b\}$:
\[
\overline{\bb{\alpha}}_{\sbullet} = \overline{\bb{\beta}}_{\sbullet} = \overline{(\bb{\alpha\beta})}_{i\sbullet} = \overline{(\bb{\alpha\beta})}_{\sbullet j} = \overline{(\bb{\alpha\beta})}_{\sbullet\sbullet} = \bb{0}_d.
\]
This is accomplished by absorbing all means into $\bb{\mu}$. The overline symbol $\overline{\bb{v}}$ over a vector $\bb{v}$ (or an array) stands for averaging and the dotted subscripts indicate over which indices the average is taken.

Extending \citet[p.~169]{MR2552961}, the corresponding MANOVA orthogonal decomposition for the total sum-of-outer-products (SOP) matrix, $\mathrm{SOP}_{\mathrm{total}}$, is
\[
\mathrm{SOP}_{\mathrm{total}} = \mathrm{SOP}_A + \mathrm{SOP}_B + \mathrm{SOP}_{AB} + \mathrm{SOP}_{\mathrm{E}},
\]
where
\[
\begin{aligned}
&\mathrm{SOP}_{\mathrm{total}} = \sum_{i=1}^a \sum_{j=1}^b \sum_{k=1}^n (\bb{Y}_{ijk} - \overline{\bb{Y}}_{\sbullet\sbullet\sbullet}) (\bb{Y}_{ijk} - \overline{\bb{Y}}_{\sbullet\sbullet\sbullet})^{\top}, \\
&\mathrm{SOP}_A = b n \sum_{i=1}^a (\overline{\bb{Y}}_{i \sbullet \sbullet} - \overline{\bb{Y}}_{\sbullet \sbullet \sbullet}) (\overline{\bb{Y}}_{i \sbullet \sbullet} - \overline{\bb{Y}}_{\sbullet \sbullet \sbullet})^{\top}, \\
&\mathrm{SOP}_B = a n \sum_{j=1}^b (\overline{\bb{Y}}_{\sbullet j \sbullet} - \overline{\bb{Y}}_{\sbullet \sbullet \sbullet}) (\overline{\bb{Y}}_{\sbullet j \sbullet} - \overline{\bb{Y}}_{\sbullet \sbullet \sbullet})^{\top}, \\
&\mathrm{SOP}_{AB} = n \sum_{i=1}^a \sum_{j=1}^b (\overline{\bb{Y}}_{ij\sbullet} - \overline{\bb{Y}}_{i\sbullet\sbullet} - \overline{\bb{Y}}_{\sbullet j\sbullet} + \overline{\bb{Y}}_{\sbullet\sbullet\sbullet}) (\overline{\bb{Y}}_{ij\sbullet} - \overline{\bb{Y}}_{i\sbullet\sbullet} - \overline{\bb{Y}}_{\sbullet j\sbullet} + \overline{\bb{Y}}_{\sbullet\sbullet\sbullet})^{\top}, \\
&\mathrm{SOP}_{\mathrm{E}} = \sum_{i=1}^a \sum_{j=1}^b \sum_{k=1}^n (\bb{Y}_{ijk} - \overline{\bb{Y}}_{ij\sbullet}) (\bb{Y}_{ijk} - \overline{\bb{Y}}_{ij\sbullet})^{\top}.
\end{aligned}
\]
The matrices $\mathrm{SOP}_A$, $\mathrm{SOP}_B$, $\mathrm{SOP}_{AB}$, and $\mathrm{SOP}_{\mathrm{E}}$ represent the ``portions'' of the total variability explained by the factors $A$, $B$, the interaction factor $AB$, and the errors, respectively.

An application of Remark~\ref{rem:normal} yields
\begin{equation}\label{eq:fixed.case.laws}
\begin{aligned}
&\mathfrak{S} \equiv \mathrm{SOP}_A \sim \mathcal{W}_d(a-1, \Sigma, \Sigma^{-1} F), \\
&\mathfrak{T} \equiv \mathrm{SOP}_B \sim \mathcal{W}_d(b-1, \Sigma, \Sigma^{-1} G), \\
&\mathfrak{U} \equiv \mathrm{SOP}_{AB} \sim \mathcal{W}_d\{(a-1)(b-1), \Sigma, \Sigma^{-1} H), \\
&\mathfrak{V} \equiv \mathrm{SOP}_{\mathrm{E}} \sim \mathcal{W}_d\{a b (n-1), \Sigma\},
\end{aligned}
\end{equation}
where the random matrices $\mathfrak{S},\mathfrak{T},\mathfrak{U},\mathfrak{V}$ are mutually independent, and
\[
F = b n \sum_{i=1}^a \bb{\alpha}_i \bb{\alpha}_i^{\top}, \quad
G = a n \sum_{j=1}^b \bb{\beta}_j \bb{\beta}_j^{\top}, \quad
H = n \sum_{i=1}^a \sum_{j=1}^b (\bb{\alpha\beta})_{ij} (\bb{\alpha\beta})_{ij}^{\top}.
\]
In this fixed-effects framework the usual questions concern equality of mean vectors. Formally, one tests
\[
\mathcal{H}_0^{A}: \boldsymbol{\alpha}_1=\dots=\boldsymbol{\alpha}_a = \bb{0}_d, \qquad
\mathcal{H}_0^{B}: \boldsymbol{\beta}_1=\dots=\boldsymbol{\beta}_b = \bb{0}_d, \qquad
\mathcal{H}_0^{AB}: (\boldsymbol{\alpha\beta})_{ij} = \bb{0}_d .
\]

Because the four matrices in~\eqref{eq:fixed.case.laws} are mutually independent central Wisharts under these null hypotheses, any test statistic that is a symmetric function of the eigenvalues of $\mathfrak{S}_{\Sigma^{-1}} (\mathfrak{V}_{\Sigma^{-1}})^{-1}$, $\mathfrak{T}_{\Sigma^{-1}} (\mathfrak{V}_{\Sigma^{-1}})^{-1}$, or $\mathfrak{U}_{\Sigma^{-1}} (\mathfrak{V}_{\Sigma^{-1}})^{-1}$ possesses a tractable null distribution. The most common choices are Wilks' lambda, Pillai's trace, the Hotelling--Lawley trace, and Roy's largest root; exact null distributions or $F$-approximations for these statistics were developed by \citet{doi:10.1093/biomet/24.3-4.471}, \citet{MR67429}, \citet{doi:10.2307/2332232}, and \citet{MR57519}, and are presented in standard references such as \citet{MR1990662}, \citet{MR560319}, and \citet{MR2962097}. Each procedure addresses the same fundamental problem: are mean responses identical across the fixed treatment groups?

The situation changes radically when some of the factors $A$, $B$, and $AB$ have random effects. Assuming, for simplicity, that they all do (the extension to the cases where some of the factors have fixed effects is trivial), the assumptions in~\eqref{eq:model} become
\[
\bb{\alpha}_i \stackrel{\mathrm{iid}}{\sim} \mathcal{N}_d(\bb{0}_d,\Sigma_{\bb{\alpha}}), \quad \bb{\beta}_j \stackrel{\mathrm{iid}}{\sim} \mathcal{N}_d(\bb{0}_d,\Sigma_{\bb{\beta}}), \quad (\bb{\alpha\beta})_{ij} \stackrel{\mathrm{iid}}{\sim} \mathcal{N}_d(\bb{0}_d,\Sigma_{\bb{\alpha\beta}}), \quad \bb{\e}_{ijk} \stackrel{\mathrm{iid}}{\sim} \mathcal{N}_d(\bb{0}, \Sigma),
\]
with mutual independence between the four sequences. In particular, this yields
\[
\begin{aligned}
&\mathfrak{S} \mid \{\mathfrak{F} = F\} \sim \mathcal{W}_d(a-1, \Sigma, \Sigma^{-1} F), \\
&\mathfrak{T} \mid \{\mathfrak{G} = G\} \sim \mathcal{W}_d(b-1, \Sigma, \Sigma^{-1} G), \\
&\mathfrak{U} \mid \{\mathfrak{H} = H\} \sim \mathcal{W}_d\{(a-1)(b-1), \Sigma, \Sigma^{-1} H\}, \\
&\mathfrak{V} \sim \mathcal{W}_d\{a b (n-1), \Sigma\}.
\end{aligned}
\]
with
\[
\begin{aligned}
&\mathfrak{F} = b n \sum_{i=1}^a \bb{\alpha}_i \bb{\alpha}_i^{\top} \sim \mathcal{W}_d(a-1, b n \Sigma_{\bb{\alpha}}), \\
&\mathfrak{G} = a n \sum_{j=1}^b \bb{\beta}_j \bb{\beta}_j^{\top} \sim \mathcal{W}_d(b-1, a n \Sigma_{\bb{\beta}}), \\
&\mathfrak{H} = n \sum_{i=1}^a \sum_{j=1}^b (\bb{\alpha\beta})_{ij} (\bb{\alpha\beta})_{ij}^{\top} \sim \mathcal{W}_d\{(a-1)(b-1), n \Sigma_{\bb{\alpha\beta}}\}.
\end{aligned}
\]
Now the scientific interest shifts to testing
\begin{equation}\label{eq:hyp}
\mathcal{H}_0^{A}: \Sigma_{\boldsymbol\alpha} = 0_{d\times d}, \qquad
\mathcal{H}_0^{B}: \Sigma_{\boldsymbol\beta} = 0_{d\times d}, \qquad
\mathcal{H}_0^{AB}: \Sigma_{\boldsymbol{\alpha\beta}} = 0_{d\times d},
\end{equation}
namely, whether the factors contribute any covariance to the multivariate response.

In this random-effects model, the between-group sums-of-outer-products become mixtures of non-central Wisharts; Wilks' $\Lambda$, Pillai's trace and the other classical MANOVA statistics thus lose their known reference distributions. While \citet{arXiv:2103.13202} proved that the univariate $F$~test remains exact in this setting when $d=1$, no multivariate analog was previously available. Theorem~\ref{thm:main} and Corollary~\ref{cor:main} fill this gap by showing that the three test statistics in~\eqref{eq:random.case.statistics} follow an exact matrix-variate Beta Type~II distribution.

Hence, the present work delivers, for the first time, finite-sample tests that detect covariance components invisible to classical mean-based MANOVA. The comparison between our new multivariate methodology and the univariate ANOVAs proposed by \citet{arXiv:2103.13202} is illustrated through two real-data examples in Sections~\ref{sec:example.1}--\ref{sec:example.2}. These approaches are best viewed as complementary, as each can reveal aspects of the data that the other may overlook.

Upon conditioning on the random effects $\bb{\alpha}_i$, $\bb{\beta}_j$, and $(\bb{\alpha\beta})_{ij}$, the model reverts to the fixed effects model \eqref{eq:model}. The conditional joint density of $(\mathfrak{S},\mathfrak{T},\mathfrak{U},\mathfrak{V})$ given $\{(\mathfrak{F},\mathfrak{G},\mathfrak{H}) = (F,G,H)\}$ is
\[
f_{\mathfrak{S},\mathfrak{T},\mathfrak{U},\mathfrak{V} \mid \mathfrak{F},\mathfrak{G},\mathfrak{H}}(S,T,U,V \mid F,G,H) = f_{\mathfrak{S} \mid \mathfrak{F}}(S \mid F) \times f_{\mathfrak{T} \mid \mathfrak{G}}(T \mid G) \times f_{\mathfrak{U} \mid \mathfrak{H}}(U \mid H) \times f_{\mathfrak{V}}(V).
\]
Integrating over the density of $(\mathfrak{F},\mathfrak{G},\mathfrak{H})$ yields the unconditional density of $(\mathfrak{S},\mathfrak{T},\mathfrak{U},\mathfrak{V})$, viz.
\[
\begin{aligned}
f_{\mathfrak{S},\mathfrak{T},\mathfrak{U},\mathfrak{V}}(S,T,U,V)
&= \int_{\mathcal{S}_{++}^d} f_{\mathfrak{S} \mid \mathfrak{F}}(S \mid F) f_{\mathfrak{F}}(F) \rd F \\[-1mm]
&\hspace{20mm}\times \int_{\mathcal{S}_{++}^d} f_{\mathfrak{T} \mid \mathfrak{G}}(T \mid G) f_{\mathfrak{G}}(G) \rd G \times \int_{\mathcal{S}_{++}^d} f_{\mathfrak{U} \mid \mathfrak{H}}(U \mid H) f_{\mathfrak{H}}(H) \rd H \times f_{\mathfrak{V}}(V) \\
&= f_{\mathfrak{S}}(S) \times f_{\mathfrak{T}}(T) \times f_{\mathfrak{U}}(U) \times f_{\mathfrak{V}}(V),
\end{aligned}
\]
so that $\mathfrak{S},\mathfrak{T},\mathfrak{U},\mathfrak{V}$ are mutually independent. Applying Corollary~\ref{cor:main} three times with
\[
\begin{aligned}
\mbox{(i)} & ~~ \mathfrak{X} = \mathfrak{S}, ~ \mathfrak{Y} = \mathfrak{F}, ~ \nu = a-1, ~ A = \Sigma, ~ \Sigma = b n \Sigma_{\bb{\alpha}}, ~ H = \Sigma^{-1}, \\
\mbox{(ii)} & ~~ \mathfrak{X} = \mathfrak{T}, ~ \mathfrak{Y} = \mathfrak{G}, ~ \nu = b-1, ~ A = \Sigma, ~ \Sigma = a n \Sigma_{\bb{\beta}}, ~ H = \Sigma^{-1}, \\
\mbox{(iii)} & ~~ \mathfrak{X} = \mathfrak{U}, ~ \mathfrak{Y} = \mathfrak{H}, ~ \nu = (a-1)(b-1), ~ A = \Sigma, ~ \Sigma = n \Sigma_{\bb{\alpha\beta}}, ~ H = \Sigma^{-1},
\end{aligned}
\]
one deduces that
\[
\begin{aligned}
\mbox{(i)} & ~~\mathfrak{S} \sim \mathcal{W}_d\{a-1, \Sigma^{1/2} (I_d + b n \Sigma^{-1/2} \Sigma_{\bb{\alpha}} \Sigma^{-1/2}) \Sigma^{1/2}\} \equiv \mathcal{W}_d(a-1, \Sigma + b n \Sigma_{\bb{\alpha}}), \\
\mbox{(ii)} & ~~\mathfrak{T} \sim \mathcal{W}_d\{b-1, \Sigma^{1/2} (I_d + a n \Sigma^{-1/2} \Sigma_{\bb{\beta}} \Sigma^{-1/2}) \Sigma^{1/2}\} \equiv \mathcal{W}_d(b-1, \Sigma + a n \Sigma_{\bb{\beta}}), \\
\mbox{(iii)} & ~~\mathfrak{U} \sim \mathcal{W}_d\{(a-1)(b-1), \Sigma^{1/2} (I_d + n \Sigma^{-1/2} \Sigma_{\bb{\alpha\beta}} \Sigma^{-1/2}) \Sigma^{1/2}\} \equiv \mathcal{W}_d\{(a-1)(b-1), \Sigma + n \Sigma_{\bb{\alpha\beta}}\}.
\end{aligned}
\]
The hypotheses $\mathcal{H}_0^A$, $\mathcal{H}_0^B$, and $\mathcal{H}_0^{AB}$ in \eqref{eq:hyp} can then be tested, respectively, using the facts that
\begin{equation}\label{eq:random.case.statistics}
\begin{aligned}
\mbox{(i)} & ~~ (\mathfrak{V}_{\Sigma^{-1}})^{-1/2} \mathfrak{S}_{\Sigma^{-1}} (\mathfrak{V}_{\Sigma^{-1}})^{-1/2} \stackrel{\mathcal{H}_0^A}{\sim} \mathcal{B}_d\{(a-1)/2, a b (n-1)/2\}, \\[-1mm]
\mbox{(ii)} & ~~ (\mathfrak{V}_{\Sigma^{-1}})^{-1/2} \mathfrak{T}_{\Sigma^{-1}} (\mathfrak{V}_{\Sigma^{-1}})^{-1/2} \stackrel{\mathcal{H}_0^B}{\sim} \mathcal{B}_d\{(b-1)/2, a b (n-1)/2\}, \\[-1mm]
\mbox{(iii)} & ~~ (\mathfrak{V}_{\Sigma^{-1}})^{-1/2} \mathfrak{U}_{\Sigma^{-1}} (\mathfrak{V}_{\Sigma^{-1}})^{-1/2} \hspace{-2mm}\stackrel{~~\mathcal{H}_0^{AB}}{\sim}\hspace{-0.5mm} \mathcal{B}_d\{(a-1)(b-1)/2, a b (n-1)/2\},
\end{aligned}
\end{equation}
where $\mathcal{B}_d(\nu_1/2,\nu_2/2)$ denotes the matrix-variate Beta Type~II distribution \citep[Definition~5.2.2]{Gupta_Nagar_1999}, for $\nu_1$, $\nu_2\in (d-1,\infty)$. It is also known as the matrix-variate $F$ distribution and corresponds to the distribution of $\smash{\mathfrak{S}_2^{-1/2} \mathfrak{S}_1 \mathfrak{S}_2^{-1/2}}$, assuming that $\mathfrak{S}_1\sim \mathcal{W}_d(\nu_1, I_d)$ and $\mathfrak{S}_2\sim \mathcal{W}_d(\nu_2,I_d)$ are independent; see, e.g., \citet[Theorem~5.2.5]{Gupta_Nagar_1999}.

\begin{remark}
If some of the factors have fixed effects, the same test statistics in \eqref{eq:random.case.statistics} can be used to test whether the corresponding factors are significant. For example, if $A$ is assumed to have fixed effects, then the hypothesis to test becomes $\mathcal{H}_0^A : F = 0_{d\times d}$, and in view of \eqref{eq:fixed.case.laws}, we still have
\[
(\mathfrak{V}_{\Sigma^{-1}})^{-1/2} \mathfrak{S}_{\Sigma^{-1}} (\mathfrak{V}_{\Sigma^{-1}})^{-1/2} \stackrel{\mathcal{H}_0^A}{\sim} \mathcal{B}_d\{(a-1)/2, a b (n-1)/2\}.
\]
\end{remark}

\subsection{First real-data example}\label{sec:example.1}

The methodology described in Section~\ref{sec:methodology} was applied to two biomarkers from the 1999--2004 US National Health and Nutrition Examination Survey (NHANES) survey \citep{doi:10.32614/CRAN.package.NHANES}: body-mass index (BMI) and total serum cholesterol (TotChol). Two survey descriptors were treated as having random effects: Education with five ordered categories (8th Grade, 9--11th Grade, High School, Some College, College Graduate) and Marital Status with six categories (Married, Live With Partner, Divorced, Separated, Widowed, Never Married). A balanced $5\times6\times5$ design was obtained by selecting $n=5$ subjects per cell, giving $N=150$ independent observations each having $d=2$ components.

The sums-of-outer-products were computed exactly as in Section \ref{sec:methodology}. For simplicity, the error covariance $\Sigma$ is taken to be any scale matrix, in which case it cancels out in the three statistics in \eqref{eq:random.case.statistics}. The three statistics were evaluated and their $p$-values estimated using $10^4$ Monte Carlo draws from a matrix-variate Beta Type II distribution, providing the correct finite-sample calibration when both factors have random effects. For comparison, univariate two-way tests of zero variance components were conducted on BMI and TotChol using the exact $F$ procedures of \citet{arXiv:2103.13202}, who proved that the usual $F$ statistic under the null retains its law under random effects in balanced designs. The results appear in Table~\ref{tab:1}.

\begin{table}[ht]
\centering
\caption{Balanced NHANES design: Education ($A$, $a=5$) and Marital Status ($B$, $b=6$) with $n=5$ observations per cell ($N=abn=150$, $d=2$). The table reports $p$‐values for $A$, $B$, and $AB$ from the multivariate and univariate approaches.}
\label{tab:1}
\begin{tabular}{lccc}
\toprule
& $p_A$ & $p_B$ & $p_{AB}$ \\\midrule
Beta Type II MANOVA on (BMI, TotChol) & 0.4993 & 0.9556 & 0.0704 \\\midrule
Variance-component $F$~test on BMI & 0.0577 & 0.2910 & 0.0008 \\
Variance-component $F$~test on TotChol & 0.3035 & 0.4935 & 0.0047 \\\bottomrule
\end{tabular}
\end{table}

Our multivariate procedure shows no evidence that Education or Marital Status alone contributes to the BMI-TotChol covariance structure, while the interaction is marginally significant at the 10\% level. In contrast, the univariate variance-component tests suggest a borderline Education effect for BMI at the 5\% level and highlight a strong Education-Marital Status interaction for both biomarkers.

While BMI and total cholesterol levels are often positively correlated, this relationship can vary across subgroups defined by education and marital status due to differences in age, sex, health-related behaviors, and other factors. Moreover, although the underlying correlation may be present, the relatively small number of observations per cell ($n = 5$) may result in a noisy empirical covariance structure within each group, making the joint signal less clearly defined. As a result, the matrix-variate Beta Type II test may have reduced power to detect interaction effects that are readily captured by separate univariate tests.

These discrepancies illustrate that multivariate covariance-based inference can reach different conclusions from univariate analyses, highlighting the importance of analyzing the joint biomarker profile alongside the individual components.

\subsection{Second real-data example}\label{sec:example.2}

To illustrate the method on a different dataset, we applied the Beta Type II MANOVA to two continuous variables from the \texttt{diamonds} dataset in the \texttt{ggplot2} package \citep{doi:10.32614/CRAN.package.ggplot2}: Carat and Price. Two categorical attributes were treated as factors having random effects: Cut (Fair, Good, Very Good, Premium, Ideal) and Color (levels from J to D in decreasing order of yellow tint).

A balanced $5\times7\times3$ design was constructed by selecting $n = 3$ diamonds per cell, for a total of $N = 105$ independent observations each having $d=2$ components. As in the first real-data example of Section~\ref{sec:example.1}, $\Sigma$ is taken to be any scale matrix and the $p$-values of the three statistics in \eqref{eq:random.case.statistics} are estimated using $10^4$ Monte Carlo draws. The results appear in Table~\ref{tab:2}.

\begin{table}[ht]
\centering
\caption{Balanced diamonds design: Cut ($A$, $a=5$) and Color ($B$, $b=7$) with $n=3$ observations per cell ($N=abn=105$, $d=2$). The table reports $p$‐values for $A$, $B$, and $AB$ from the multivariate and univariate approaches.}
\label{tab:2}
\begin{tabular}{lccc}
\toprule
& $p_A$ & $p_B$ & $p_{AB}$ \\\midrule
Beta Type II MANOVA on (Carat, Price) & 0.0000 & 0.0000 & 0.0153 \\\midrule
Variance-component $F$~test on Carat & 0.0000 & 0.0526 & 0.0056 \\
Variance-component $F$~test on Price & 0.0000 & 0.0054 & 0.0115 \\\bottomrule
\end{tabular}
\end{table}

In this example, the multivariate analysis reveals highly significant effects of Cut and Color on the joint distribution of Carat and Price, as well as a notable interaction effect between the two factors. While the univariate tests also detect strong effects overall, their conclusions are less uniform; for instance, the effect of Color on Carat is only borderline significant at the 5\% level. In contrast, the multivariate procedure detects a more pronounced overall influence of Color when both variables are considered simultaneously, highlighting a joint structure not fully captured by the marginal tests.

This distinction underscores the value of multivariate covariance-based inference in uncovering interactions and main effects that may be diluted or obscured in separate univariate analyses. Compared to the NHANES setting in Section~\ref{sec:example.1}, where multivariate signals were weaker, the present example more clearly illustrates the capacity of the matrix-variate Beta Type II method to detect meaningful factor effects on multivariate outcomes.

\section*{Reproducibility}

The \textsf{R} codes that produced the results of the real-data examples in Sections~\ref{sec:example.1}--\ref{sec:example.2} are publicly available in the GitHub repository of \citet{GenestMacKayOuimet2025github}.

\section*{Acknowledgments}
\addcontentsline{toc}{section}{Acknowledgments}

We thank Donald Richards (Penn State University) for enlightening discussions.

\section*{Funding}
\addcontentsline{toc}{section}{Funding}

Genest's work was funded by the Natural Sciences and Engineering Research Council of Canada (Grant RGPIN-2024-04088) and the Canada Research Chairs Program (Grant 950-231937). MacKay's work was funded by the Natural Sciences and Engineering Research Council of Canada (Grant RGPIN-2024-05794). The first version of this work was completed while Ouimet was a Research Associate at McGill University and a postdoctoral fellow at the Universit\'e de Sherbrooke. These positions were funded through Genest's and MacKay's research grants, respectively.

\addcontentsline{toc}{chapter}{References}

\bibliographystyle{authordate1}
\bibliography{GMO_bib}

\end{document}